\documentclass{amsart}

\usepackage{allneeded}

\begin{document}

\title[Non-iterated periodic billiard trajectories]{On the multiplicity of non-iterated\\ periodic billiard trajectories}%

\author{Marco Mazzucchelli}%

\address{Department of Mathematics, Penn State University, University Park, PA 16802, USA}%
\email{mazzucchel@math.psu.edu}%


\date{October 3, 2010. \emph{Revised}: January 19, 2011}%

\maketitle
\begin{abstract}
We introduce the iteration theory for periodic billiard trajectories in a compact and convex domain of the Euclidean space, and we apply it  to establish a multiplicity result for non-iterated trajectories.
\end{abstract}
\begin{quote}
\begin{footnotesize}
\tableofcontents
\end{footnotesize}
\end{quote}

\section{Introduction}

Billiard dynamics describes the motion of a particle moving without friction in a compact domain of $\R^{N+1}$, for $N\geq1$, while being subject to a singular potential which is identically zero in the interior of the domain and $+\infty$ on the boundary. This implies that the particle moves on straight lines with constant speed until it reaches the boundary of the domain, where it bounces with specular reflection and with no loss of energy. In this paper, we investigate the multiplicity of certain periodic billiard trajectories in strictly convex domains of $\R^{N+1}$ enclosed by a smooth hypersurface, which is therefore diffeomorphic to the unit sphere $S^N$. Throughout the paper, by strict convexity we will always mean that the second fundamental form of the boundary is everywhere positive definite.

Historically, the first multiplicity result for periodic billiard trajectories was proved by Birkhoff \cite{Bi} for convex plane (i.e.\ $N=1$) billiards.  The result asserts that, for each pair of coprime positive integers $n\geq2$ and $r\leq n/2$, there are at least two distinct periodic billiard trajectories with $n$ bounce points and rotation number $r$. The plane case is special, since the billiard dynamics can be described by means of an area preserving twist-map on the annulus, and nowadays Birkhoff's Theorem can be proved by means of Aubry-Mather theory, see~\cite{Ba88} and references therein. 
For higher dimensional billiards the problem is essentially harder, and estimates for the number of bounce trajectories with prescribed number of bounce points have been proved by Farber and Tabachnikov~\cite{FaTa, FaTa02b}, who corrected and  extended an earlier proof by Babenko~\cite{Bab}. Further extensions, together with more sophisticated multiplicity results, have been proved by Farber~\cite{Fa} and by Karasev~\cite{Kar09}. Roughly, these results can be summarized by saying that, for each $n$ odd, the number of periodic billiard trajectories whose number of bounce points divides $n$ grows at least linearly in $n$. 

The arguments in these papers are based on critical point theory. In fact, billiard trajectories can be characterized by a variational principle, which in the periodic case goes as follows. If $S$ is the smooth boundary of our convex billiard table and $n\in\N$, the length functional  $\length n$, defined on the $n$-fold cross product of $S$, computes the perimeter of the $n$-gon inscribed by a given sequence of points $\qq=(q_0,...,q_{n-1})\in S^{\times n}$. This functional is clearly continuous, and actually smooth when restricted to the so-called cyclic configuration space $\conf n(S)$, the space of  those $\qq$'s such that $q_j\neq q_{j+1}$ for all $j\in\Z_n$. The critical points $\qq\in\conf n(S)$ of $\length n$ are precisely the $n$-periodic sequences of bounce points of billiard trajectories. Notice that the dihedral group $\D_n$ acts by permutations on the cyclic configuration space $\conf n(S)$, and all the points that belong to a same orbit of its action refer to the same geometric closed curve.

Now, the lack of compactness of $\conf n(S)$ does not represent a real obstacle for applying the machinery of critical point theory with the length functional $\length n$ or, more precisely, with the functional $-\length n$ (see~\cite[Section~4]{FaTa} or the paragraph after Proposition~\ref{p:FaTa}). In particular, each $\D_n$-equivariant homology or cohomology class of   $\conf n(S)$ gives rise to a billiard periodic trajectory (as critical point of $\length n$) and, at least in the non-degenerate case, linearly independent classes produce different critical points. In~\cite{FaTa}, the equivariant cohomology algebra $\Hom^*_{\D_n}(\conf n(S);\Z_2)$ has been computed for each $n$ odd. Its rank  gives, in the non-degenerate case, a lower bound for the number of periodic billiard trajectories whose number of bounce points divides $n$. In the degenerate case, by Lusternik-Schnirelmann theory, such a lower bound is given by the cup-length of $\Hom^*_{\D_n}(\conf n(S);\Z_2)$ plus one.

In this paper, we proceed along a different line: the result we prove concerns the multiplicity of periodic orbits whose number of bounce points lies in the set of powers of a given prime number $p$. Our proof will use a minimum amount of information on the homology of the configuration space: we will only need that,  for each $n\in\N$, the homology group $\Hom_{*}(\conf n(S);\Z_2)$ is nontrivial in degree $N-1$, where $N=\dim(S)$. Via Morse theory, for each $n\in\N$, a nonzero homology class in $\Hom_{N-1}(\conf n(S);\Z_2)$ generates a periodic billiard trajectory $\gamma_n$ whose number of bounce points divides $n$. More precisely, if the number of its bounce points is $n/m$, the generated critical point $\qq_n=(q_{n,0},...,q_{n,n})\in\conf n(S)$ of $\length n$ is the sequence of bounce points of  the $m$-fold iteration of $\gamma_n$. Since we only make use of the homology of $\conf n(S)$ in degree $N-1$, all the $\gamma_n$'s (i.e.\ all the associated critical points of the length functionals $\length n$) have Morse coindex less than or equal to $N-1$ and are not local maxima for $\length n$.

The main issue here is to prove that, varying $n$, we obtain infinitely many distinct trajectories $\gamma_{p^n}$. In order to prove this, we develop an iteration theory for billiard periodic trajectories, a discrete version of the one for closed geodesics (see e.g.~\cite{Bo, GM_geod, BK}), that may also have independent interest. More specifically, we investigate the behavior of the Morse indices and of the local homology of billiard periodic trajectories under iteration, and we prove a discrete version of Bangert and Klingenberg's homological vanishing under iteration, a tool that can be used in certain situations to assert the existence of infinitely many closed geodesics.

The main result of the paper is the following.

\begin{thm}\label{t:main}
Consider the billiard dynamics inside a strictly  convex subset enclosed by a smooth hypersurface of $\R^{N+1}$, where $N\geq 2$. 
For each prime number $p$ at least one of the following two statements is satisfied: 
\begin{itemize}
\item there is some $n\in\N$ such that infinitely many periodic billiard trajectories bounce $p^{n}$ times;
\item there is a sequence $\{\gamma_\alpha\,|\,\alpha\in\N\}$ of geometrically distinct periodic billiard trajectories such that:
\begin{itemize}
\item[(i)] each $\gamma_\alpha$ bounces $p^{n_\alpha}$ times for some $n_\alpha\in\N$, 
\item[(ii)] each $\gamma_\alpha$ has Morse coindex less than or equal to $N$ and it is not a local maximum for the length functional $\length{p^{n_\alpha}}$.
\end{itemize}
\end{itemize}
\end{thm}

We stress that Theorem~\ref{t:main} does not follow from the mere knowledge of the cohomology of the cyclic configuration space (and thus, it does not follow from \cite[Theorem~1]{FaTa}). In fact, the Poincar\'e polynomial associated to $\Hom^*_{\D_n}(\conf n(S);\Z_2)$ is given by
\begin{align*}
P_{N,n}(t)
&=
\frac{(t^{(n-1)(N-1)}-1)(t^N-1)(t^N+1)}{(t^{2(N-1)}-1)(t-1)}\\
&=
\cbra{
\sum_{j=0}^{(n-3)/2}
t^{2(N-1)j}
}
\cbra{
\sum_{k=0}^{N-1}
t^{k}
}
(t^N+1)
\end{align*}
provided $N=\dim(S)\geq3$ and $n$ is odd, see\footnote{In~\cite[Theorem~7]{FaTa} the formula of the Poincar\'e polynomial contains a typo.} \cite[Theorem~7 and Proof of Proposition~4.5]{FaTa}. Even in the non-degenerate case, by $\D_n$-equivariant Morse theory, a lower bound for the number of billiard periodic trajectories having Morse coindex less than or equal to $N$ and number of bounce points that divides $n$ is only given by
\begin{align*}
\sum_{j=0}^N
\dim
\Hom^j_{\D_n}(\conf n(S);\Z_2)
&=
P_{N,n}(0)+
P_{N,n}'(0)+
\sfrac{1}{2} P_{N,n}''(0)+
...
+
\sfrac{1}{N!} P_{N,n}^{(N)}(0)\\
&=
N+1.
\end{align*}
The iteration theory developed in the current paper  is thus needed to conclude that, by varying $n$ in the set of powers of any given prime number, the $n$-periodic billiard  trajectories that are found by Morse theory are not all iterations of a finite numbers of lower periodic  ones.

Beside the considerations on the Morse coindex, by choosing $p=2$ in Theorem~1.1 we obtain that every convex billiard table admits infinitely many (geometrically distinct) periodic billiard trajectories whose numbers of bounce points are powers of $2$. To the best of the author's knowledge, this assertion does not follows from any known multiplicity result in the literature.

\subsection{Organization of the paper}
In section~\ref{s:preliminaries} we recall the basic definitions concerning billiards and the variational principle for periodic billiard trajectories. In section~\ref{s:iter} we introduce the iteration theory for periodic billiard trajectories: in subsection~\ref{s:iteration_Morse} we discuss the behavior of the Morse indices under iteration, while  in subsection~\ref{ss:loc_hom} and~\ref{ss:hom_vanish} we investigate the behavior of the local homology groups of periodic billiard trajectories under iteration. Section~\ref{s:proof} is devoted to the proof of Theorem~\ref{t:main}.

\subsection{Acknowledgments} The author wishes to thank Alberto Abbondandolo,  Sergei Tabachnikov and the anonymous referee for encouraging  and for useful remarks on a preliminary version of the paper. This research has been supported by the Max Planck Institute for Mathematics in the Sciences (Leipzig, Germany) and by the ANR project ``KAM faible''.


\section{Preliminaries}\label{s:preliminaries}

Throughout this paper, $S$ will be a smooth  hypersurface in $\R^{N+1}$, where $N\in\N=\{1,2,3,...\}$, enclosing a  compact and strictly convex domain $U_S$. Hereafter, strict convexity must be intended in the differentiable sense: the second fundamental form of $S$ is everywhere positive definite. In particular, $S$ is diffeomorphic to the unit $N$-sphere $S^N\subset\R^{N+1}$. We are interested in the billiard dynamics in  $U_S$. A curve $\gamma$ inside $U_S$ is a \textbf{billiard trajectory} if it is a piecewise straight curve with constant speed $|\dot\gamma|$, except at the instants $t$ in which it hits the hypersurface $S$, where it bounces according to the usual law of reflection: the component of the velocity that is normal to the boundary instantaneously changes sign, whereas the tangential component is preserved. A billiard trajectory $\gamma$ is \textbf{periodic} with period $T>0$ if it is a curve of the form $\gamma:\R/T\Z\to U_S$.

Billiard periodic orbits are characterized by a well-known variational principle which we are going to recall. For each $n\in\N$, let us denote by $S^{\times n}$ the $n$-fold product $S\times ...\times S$. We consider the open subset $\conf n(S)\subset S^{\times n}$ given by 
\[ 
\conf n(S)=\gbra{ \qq=(q_0,...,q_{n-1})\in S^{\times n}\,|\, q_j\neq q_{j+1}\ \forall j\in\Z_n }, 
\]
which we will refer to as the \textbf{cyclic configuration space} (or simply the \textbf{configuration space}). The dihedral group $\D_n$, seen as a group of permutations of $\Z_n$, acts on $\conf n(S)$ by
\begin{align*}
\sigma\cdot(q_0,...,q_{n-1})=(q_{\sigma(0)},...,q_{\sigma(n-1)}),\s\s\forall\qq\in\conf n(S), \sigma\in\D_n. 
\end{align*} 
For each $\qq\in\conf n(S)$, we denote by $\gamma_{\qq}$ the unique curve in $U_S$ with prescribed speed (say, parametrized by arc-length) that is piecewise straight and bounces periodically in the points $q_0,...,q_{n-1}$ (see Figure~\ref{f:notation}). Notice that each point in the $\D_n$-orbit of $\qq$ is associated to the same geometric curve $\gamma_{\qq}$.

\frag[s]{S}{$S$}%
\frag[s]{0}{$q_0$}%
\frag[s]{1}{$q_1$}%
\frag[s]{2}{$q_2$}%
\frag[s]{3}{$q_3$}%
\frag[s]{g}{$\gamma_{\qq}$}%

\begin{figure}
\begin{center}
\includegraphics{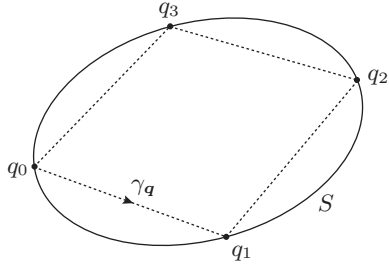}
\end{center}
\captionstyle{myCenter}
\caption{Closed curve $\gamma_{\qq}$ and corresponding sequence of bounce points $\qq=(q_0,...,q_3)$.}
\label{f:notation}
\end{figure}

We denote by $\length n:S^{\times n}\to \R$ the \textbf{length functional} de\-fined by
\[ \length n(\qq)=\sum_{j\in\Z_n} |q_{j+1}-q_j|,\s\s\forall \qq=(q_0,...,q_{n-1})\in S^{\times n},  \]
namely, $\length n(\qq)$ is the length of the closed curve $\gamma_{\qq}$.  This functional is continuous, and  it is smooth on the configuration space $\conf n(S)$, whereas  is not even differentiable in the complement  of $\conf n(S)$ (since norms are not differentiable at the origin). A straightforward computation shows that
\begin{align*}
\diff \length n(\bm q)\,\bm v
=
\sum_{j\in\Z_n}
\langle
\underbrace{\sfrac{q_j-q_{j-1}}{|q_j-q_{j-1}|}
-
\sfrac{q_{j+1}-q_j}{|q_{j+1}-q_j|}}_{q_j'}
,
v_j
\rangle,\s\s
\forall\qq\in\conf n(S), \vv\in\Tan_{\qq}S^{\times n}.
\end{align*}
A point $\qq\in\conf n(S)$ is then a critical point of the length functional $\length n$ if and only if, for each $j\in\Z_n$, the associated point $q_j'\in\R^{N+1}$ is orthogonal to the tangent space $\Tan_{q_j}S$. This amounts to requiring that the curve $\gamma_{\qq}$ satisfies the  reflection law at the bounce points, and therefore $\qq$ is a critical point of $\length n$ if and only $\gamma_{\qq}$ is a  periodic billiard trajectory.

Having this variational principle, one is tempted to study the multiplicity of periodic billiard trajectories with a prescribed number $n$ of bounce points by means of critical point theory, more specifically by means of Morse theory or Lusternik-Schnirelmann theory for the length functional $\length n:\conf n(S)\to\R$. However, one immediately faces the problem of the lack of compactness of the configuration space $\conf n(S)$. A possible solution has been suggested by Farber and Tabachnikov \cite{FaTa}, who extended earlier works in \cite{CrSw, Bab, KoTr} for the two-dimensional case. The idea is to restrict $\length n$ to the compact subspace $\cconf n\epsilon(S)\subset\conf n(S)$, where $\epsilon>0$ and 
\begin{align}\label{e:cconf}
\cconf n\epsilon(S)=\biggl\{ \qq\in\conf n(S)\,\biggr|\, \prod_{j\in\Z_n} |q_j-q_{j-1}|\geq\epsilon^n  \biggr\},
\end{align}
by virtue of the following statement.

\begin{prop}[Proposition~4.1 in \cite{FaTa}]\label{p:FaTa}
For each sufficiently small $\epsilon>0$ the following claims hold:
\begin{itemize}
\item[(i)] $\cconf n\epsilon(S)$ is a smooth manifold with boundary;
\item[(ii)] the inclusion $\cconf n\epsilon(S)\subset\conf n(S)$ is a homotopy equivalence;
\item[(iii)] all the critical points of $\length n:\conf n(S)\to\R$ are contained in $\cconf n\epsilon(S)$;
\item[(iv)] at every point of $\partial\cconf n\epsilon(S)$, the gradient of $\length n$ points inward.

\hfill$\qed$
\end{itemize}
\end{prop}
This proposition guarantees that, for any $\epsilon>0$ sufficiently small, the functional $\length n:\cconf n\epsilon(S)\to\R$ satisfies the so-called ``general boundary conditions'' for Morse theory, see \cite[Section~6.1]{Ch}. Moreover, if we perform Morse theory using the gradient flow of $\length n$ (instead of its anti-gradient flow, as it would be more common), point (iv) of the proposition implies that the boundary of $\cconf n\epsilon(S)$ does not enter into play while applying the principles of Morse theory.

The homology and cohomology of cyclic  configuration spaces have been studied by many authors, see e.g.~\cite{Ar, Coh1, Coh2, FaTa}. In particular, in~\cite{FaTa} the cohomology and the $\D_n$-equivariant cohomology rings of $\conf n(S)$ with $\Z_2$ coefficients have been completely determined. For our purposes, we only need to recall that the \textbf{Poincar\'e polynomial} of $\Hom_*(\conf n(S);\Z_2)$, namely the polynomial
\begin{align*}
B_{N,n}(t)=\sum_{j=0}^{nN} \dim \Hom_j(\conf n(S);\Z_2)\,t^j,
\end{align*}
is equal to
\begin{align}\label{e:Poincare_polynomial}
B_{N,n}(t)
&=
\frac{(t^N+1)(t^{(n-1)(N-1)}-1)}{t^{N-1}-1}
=
(t^N+1)
\cbra{
\sum_{j=0}^{n-2}
t^{(N-1)j}
}
,
\end{align}
see \cite[Theorem~4 and Remarks~3.2 and~3.3]{FaTa}.


\section{Iteration theory for  periodic billiard  trajectories}\label{s:iter}

\subsection{Morse indices of iterated periodic billiard  trajectories}\label{s:iteration_Morse}

For each $n,m\in\N$, consider the embedding $\itmap m:S^{\times n}\hookrightarrow S^{\times nm}$ given by
\[ \itmap m(\qq)=\qq\iter m:=(\underbrace{\bigl.\qq,...,\qq}_{\times m}),\s\s\forall \qq\in S^{\times n}. \]
This map is clearly smooth (being the restriction of the linear ``diagonal'' embedding $\R^{nN}\hookrightarrow\R^{nmN}$), and it restricts as a map $\conf n(S)\hookrightarrow\conf {nm}(S)$ that we will still denote by $\itmap m$. This latter map has a clear interpretation in terms of piecewise straight closed curves associated to the points of the configuration spaces: if $\gamma_{\qq}$ and $\gamma_{\qq\iter m}$ are the closed curves associated to $\qq$ and $\qq\iter m$ respectively, then $\gamma_{\qq\iter m}$ is the $m$-fold iterate of $\gamma_{\qq}$. For this reason we call  $\itmap m$ the ($m$-fold) \textbf{iteration map}.

By the characterization of the critical points of the length functional as bounce points of  periodic billiard trajectories, it is clear that $\qq\in\conf n(S)$ is a critical point of $\length n$ if an only if, for some (and thus for all) $m\in\N$, its iteration $\qq\iter m$ is a critical point of $\length{nm}$. We denote by $\ind(\qq)$, $\coind(\qq)$ and $\nul(\qq)$ the Morse index, the Morse coindex and the nullity of the length functional $\length n$ at $\qq$. We recall that these are nonnegative  integers  defined respectively as the dimension of the negative eigenspace, of the positive eigenspace and of the kernel of the Hessian of $\length n$ at $\qq$. In this section we investigate the properties of the sequences $\{\ind(\qq\iter m)\,|\,m\in\N\}$, $\{\coind(\qq\iter m)\,|\,m\in\N\}$ and $\{\nul(\qq\iter m)\,|\,m\in\N\}$. This of course is reminiscent of the iteration theory for the Morse indices of closed geodesics, which has  essentially been pioneered by Bott in his celebrated paper \cite{Bo}, and further extended by many authors also to more general Maslov-type indices, see \cite[Part~IV]{Lo_book} and the bibliography therein for a detailed account. The results that we are going to present can be considered as discrete-time versions of those contained in \cite[Section~1]{Bo}. In view of the application to the multiplicity of periodic billiard trajectories in section~\ref{s:proof} we draw, as a consequence of this iteration theory, the following iteration inequalities analogous to the one established by Liu and Long in \cite{LLo1,LLo2} (see also \cite[page~213]{Lo_book}) for a Maslov-type index.

\begin{prop}[Iteration inequalities]
\label{p:iteration_inequalities}
For each critical point $\qq$ of $\length n$, the following claims hold.
\begin{itemize}

\item[(i)] The nullity of the iterations of $\qq$ is  uniformly bounded by $2N$, i.e.
\[\nul(\qq\iter m)\leq2N.\]

\item[(ii)] There is a nonnegative real number $\avind(\qq)$, the \textnormal{\textbf{mean Morse index}} of $\length n$ at $\qq$, defined by
\[ 
\avind(\qq)=\lim_{m\to\infty} \sfrac1m\, \ind(\qq\iter m),\s\s
\]
such that the following inequalities are verified:
\begin{align*}
m\,\avind(\qq) - 2N  \leq   \ind(\qq\iter m) \leq
m\,\avind(\qq) + 2N - \nul(\qq\iter m) .
\end{align*}
Analogously, there is a nonnegative real number $\avcoind(\qq)$, the \textnormal{\textbf{mean Morse coindex}} of $\length n$ at $\qq$, defined by
\[ 
\avcoind(\qq)=\lim_{m\to\infty} \sfrac1m\, \coind(\qq\iter m),\s\s
\]
such that the following inequalities are verified:
\begin{align*}
m\,\avcoind(\qq) - 2N  \leq    \coind(\qq\iter m) \leq
m\,\avcoind(\qq) + 2N - \nul(\qq\iter m).
\end{align*}
\end{itemize}
\end{prop}
The proof of this proposition will be carried out at the end of this subsection. The reader might skip the following paragraphs and go directly to subsection~\ref{ss:loc_hom} on a first reading.

To begin with, let us write down an expression for the Hessian of $\length n$ at the critical point $\qq$. For each $\vv,\ww\in\Tan_{\qq}S^{\times n}$ we have
\begin{align*}
\hess\length n(\qq)[\vv,\ww]
=
&
\sum_{j\in\Z_n}
\sfrac{1}{|q_{j+1}-q_j|}
\Bigl(
\langle
v_{j+1}-v_j,
w_{j+1}-w_j
\rangle \\
&
-
\langle
\sfrac{q_{j+1}-q_j}{|q_{j+1}-q_j|},
v_{j+1}-v_j
\rangle
\,
\langle
\sfrac{q_{j+1}-q_j}{|q_{j+1}-q_j|},
w_{j+1}-w_j
\rangle
\Bigr).
\end{align*}
Let us denote by $H=H_{\qq}$ the self-adjoint  endomorphism of $\Tan_{\qq}S^{\times n}$ associated to the Hessian of $\length n$ at $\qq$, i.e.\ $\hess\length n(\qq)[\vv,\ww]=\langle H\vv,\ww \rangle$. If we write \[H\vv=((H\vv)_1,...,(H\vv)_n),\] then for each $j\in\Z_n$ we have
\begin{align}
\label{e:H}
(H\vv)_j
=
-\sfrac{1}{|q_{j+1}-q_j|} \pi_j\circ\tilde\pi_j(v_{j+1}-v_j)
+
\sfrac{1}{|q_j-q_{j-1}|} \pi_j\circ\tilde\pi_{j-1}(v_j-v_{j-1}),
\end{align}
where $\pi_j:\R^{N+1}\to\Tan_{q_j}S$ and $\tilde\pi_j:\R^{N+1}\to \langle q_{j+1}-q_j\rangle^\bot$ are orthogonal projectors. This expression shows that $H$ is a second order difference operator. Now, for each $m\in\N$ and $z\in S^1\subset\C$, we consider the vector space of sequences $\V_{m,z}$ given by
\[
\V_{m,z}=\gbra{ \nnu=\{\nu_j\,|\,j\in\Z\}\,|\,\nu_j\in\Tan_{q_j}S\otimes\C,\ \nu_{j+nm}=z\,\nu_j\ \ \forall j\in\Z },
\]
and, for each $\lambda\in\R$, the eigenvalue problem
\begin{align}\label{e:eigenvalue}
\left\{
  \begin{array}{l}
    H\nnu=\lambda \nnu, \\ 
    \nnu\in\V_{m,z}. \\ 
  \end{array}
\right.
\end{align}
We denote by $\ind_{z,\lambda}(\qq\iter m)$ the number of complex-linearly independent solutions of this eigenvalue problem. Notice that, since $H$ is a real operator with real eigenvalues, the sequence $\nnu\in\V_{m,z}$ is a solution of $H\nnu=\lambda\nnu$ if and only of the complex conjugate sequence $\bar\nnu\in\V_{m,\bar z}$ is a solution of $H\bar\nnu=\lambda\bar\nnu$, and therefore
\[\ind_{z,\lambda}(\qq\iter m)=\ind_{\bar z,\lambda}(\qq\iter m).\]
Now, let us set
\begin{align*}
\ind_z(\qq\iter m)&:=\sum_{\lambda<0} \ind_{z,\lambda}(\qq\iter m),\\
\coind_z(\qq\iter m)&:=\sum_{\lambda>0} \ind_{z,\lambda}(\qq\iter m),\\
\nul_z(\qq\iter m)&:=\ind_{z,0}(\qq\iter m)
\end{align*}
The above sums are finite, for $\ind_{z,\lambda}(\qq\iter m)$ is different from zero only if $\lambda$ belongs to the (finite) spectrum of $H$. These integer indices that we have just defined generalize the Morse index, the Morse coindex and the nullity for, as it readily follows from their definition, we have
\begin{align*}
\ind(\qq\iter m)&=\ind_1(\qq\iter m),\\
\coind(\qq\iter m)&=\coind_1(\qq\iter m),\\
\nul(\qq\iter m)&=\nul_1(\qq\iter m).
\end{align*}

The reason for considering the eigenvalue problem for the operator $H$ in the complexified setting is that there is a nice way to compute the index $\ind_{z,\lambda}(\qq\iter m)$ from the indices $\ind_{w,\lambda}(\qq)$, for every $w\in S^1$, of the non-iterated critical point. The recipe is given by the following statement which is analogous to \cite[Theorem~I]{Bo}. Its proof is a simple application of the Fourier expansion of ``periodic'' sequences, and we include it here for the reader's convenience.

\begin{prop}[Theorem~I in \cite{Bo}]\label{p:Fourier}
For each  $z\in S^1$, $\lambda\in\R$ and $m\in\N$, the index $\ind_{z,\lambda}(\qq\iter m)$ satisfies
\[ 
\ind_{z,\lambda}(\qq\iter m) 
= 
\sum_{w\in \sqrt[m]{z}}
\ind_{w,\lambda}(\qq). 
\]
\begin{proof}
The operator $H$ maps each vector space $\V_{m,z}$ to itself. In fact, consider the shift operator $S:\V_{m,z}\to\V_{m,z}$ given by 
\[
(S\nnu)_j
=
\nu_{j+n},\s\s
\forall \nnu=\gbra{\nu_j\,|\,j\in\Z}\in\V_{m,z}.
\]
By the definition of $H$ (in particular, notice that the index $j$ in equation~\eqref{e:H} is defined modulo $n$) we have that $SH=HS$. Moreover, since $H$ is a real operator, we have that 
\[
zH\nnu=Hz\nnu=HS^m\nnu=S^mH\nnu,\s\s\forall\nnu\in\V_{m,z},
\] 
which proves the claim.

Now, every $\nnu=\gbra{\nu_j\,|\,j\in\Z}\in\V_{m,z}$ admits a unique Fourier expansion
\[ \nnu = \sum_{w\in\sqrt[m]{z}} \nnu_{w}  \]
where, for each $w\in\sqrt[m]{z}$, the sequence $\nnu_w=\{\nu_{w,j}\,|\,j\in\Z\}$ belongs to the vector space $\V_{1,w}$. One can explicitly compute $\nnu_w$ as
\[ \nu_{w,j} = \frac1m \sum_{h=0}^{m-1} w^{1-h} \nu_{j+h},\s\s \forall j\in\Z.  \]
By the first part of the proof we have that $H\nnu\in\V_{m,z}$ and $H\nnu_w\in\V_{1,w}$. Hence, the unique Fourier expansion of $H\nnu\in\V_{m,z}$ is given by
\[ H\nnu=\sum_{w\in\sqrt[m]{z}} H\nnu_{w}. \]
From this, we conclude that $\nnu$ satisfies the eigenvalue problem $H\nnu=\lambda\nnu$ if and only if all the $\nnu_w$'s satisfy the same eigenvalue problem $H\nnu_w=\lambda\nnu_w$.
\end{proof}
\end{prop}

This proposition tells us that it is enough to study the indices of $\qq\iter m$ with a fixed $m\in\N$, say $m=1$.  To start with, let us investigate the properties of $\nul_z(\qq)$. We call $z\in S^1$ a \textbf{Poincar\'e point} of $\qq$ when $\nul_z(\qq)\neq0$.

\begin{prop}\label{p:Poincare}
There are only finitely many Poincar\'e points $z_1,...,z_r\in S^1$ and we have
\[ \sum_{\alpha=1}^r \nul_{z_\alpha}(\qq) \leq 2N. \]
\begin{proof}
Let $\vv=\gbra{ v_j}$ be a sequence such that, for each $j\in\Z$, the element $ v_j$ belongs to the  tangent space $\Tan_{q_j}S$. By~\eqref{e:H}, $\vv$ satisfies $H\vv=0$ if and only if 
\begin{align}
\label{e:Hnu}
\sfrac{1}{|q_{j+1}-q_j|} \pi_j\circ\tilde\pi_j( v_{j+1}- v_j)
=
\sfrac{1}{|q_j-q_{j-1}|} \pi_j\circ\tilde\pi_{j-1}( v_j- v_{j-1}),\s\s\forall z\in\Z,
\end{align}
where  $\pi_j:\R^{N+1}\to\Tan_{q_j}S$ and $\tilde\pi_j:\R^{N+1}\to \langle q_{j+1}-q_j\rangle^\bot$  are orthogonal projectors as above. For our convenience, let us rewrite equation~\eqref{e:Hnu} as
\begin{align}\label{e:Hnu2}
\pi_j\circ\tilde\pi_j( v_{j+1})
=
\pi_j\circ\tilde\pi_j( v_j)
+
\sfrac{|q_{j+1}-q_j|}{|q_j-q_{j-1}|} \pi_j\circ\tilde\pi_{j-1}( v_j- v_{j-1}).
\end{align}
Now, since the vector $q_{j+1}-q_j\in\R^{N+1}$ is transverse to $\Tan_{q_{j+1}}S$ (as well as to $\Tan_{q_{j}}S$), the composition $\pi_j\circ\tilde\pi_j$ restricts to an isomorphism 
\[ 
\Tan_{q_{j+1}}S  \toup^{\cong} \Tan_{q_{j}}S. 
\]
This shows that we can rewrite equation~\eqref{e:Hnu2} as 
\begin{align} 
\label{e:Hnu_expl1}
 v_{j+1}= A_j  v_j + B_j  v_{j-1},
\end{align}
where $A_j:\Tan_{q_{j}}S \to\Tan_{q_{j+1}}S$ and $B_j:\Tan_{q_{j-1}}S \to\Tan_{q_{j+1}}S$ are  linear maps. Analogously, we can rewrite equation~\eqref{e:Hnu2} as 
\begin{align} 
\label{e:Hnu_expl2}
  v_{j-1}= C_j  v_{j+1} + D_j  v_{j},
\end{align}
where $C_j:\Tan_{q_{j+1}}S  \to\Tan_{q_{j-1}}S$ and $D_j:\Tan_{q_{j}}S \to\Tan_{q_{j-1}}S$ are  linear maps. Equations~\eqref{e:Hnu_expl1} and~\eqref{e:Hnu_expl2} show that every solution $\vv=\{ v_j\}$ of $H\vv=0$ is completely determined by two of its subsequent points, say $( v_0, v_1)$, and conversely any choice of $( v_0, v_1)$ uniquely determine a solution $\vv$. Moreover, $\vv$ depends linearly on $( v_0, v_1)$. Let us denote by $\Phi$ the linear endomorphism of $\Tan_{q_0}S\oplus\Tan_{q_1}S$ given by $\Phi(v_0,v_1)=(v_n,v_{n+1})$, and  let us extend it as a complex linear endomorphism of $(\Tan_{q_0}S\oplus\Tan_{q_1}S)\otimes\C$. 

Now, let us take $z\in S^1$. By its definition, the integer $\nul_{z}(\qq)$ is equal to  the complex dimension of the kernel of $(\Phi-z\mathrm{Id})$. This establishes the proposition.
\end{proof}
\end{prop}

\begin{rem}\label{r:Poincare}
If $z_1,...,z_r$ are the Poincar\'e points of $\qq$,  their $m^{\mathrm{th}}$ powers $z_1^m,...,z_r^m$ are the Poincar\'e points of $\qq\iter m$, for each $m\in\N$.
\hfill\qed
\end{rem}

The next statement summarizes the properties of $\ind_{z}(\qq)$ and $\coind_z(\qq)$. 
\begin{prop} \label{p:ind_z,coind_z}
The functions $z\mapsto \ind_{z}(\qq)$ and $z\mapsto \coind_{z}(\qq)$ are locally constant on $S^1\setminus\{z_1,...,z_r\}$, where $z_1,...,z_r$ are the Poincar\'e points, and lower semi-continuous on $S^1$. Moreover, the jump of these functions at any  Poincar\'e point $z_\alpha$ is bounded in absolute value by $\nul_{z_\alpha}(\qq)$, i.e.
\begin{gather*}
\ind_{z_\alpha}(\qq) \leq \lim_{z\to z_\alpha^\pm} \ind_{z}(\qq) \leq \ind_{z_\alpha}(\qq) + \nul_{z_\alpha}(\qq),\\
\coind_{z_\alpha}(\qq)\leq\lim_{z\to z_\alpha^\pm} \coind_{z}(\qq)\leq \coind_{z_\alpha}(\qq) + \nul_{z_\alpha}(\qq). 
\end{gather*}
\begin{proof}
For each $z\in S^1$, let us denote by $\sigma_z\subset\R$ the spectrum of the operator $H:\V_{1,z}\to\V_{1,z}$. This spectrum satisfies the following continuity property: for each interval $(a,b)\subset\R\cup\{\pm\infty\}$ such that $a$ and $b$ do not belong to $\sigma_z$, there is a neighborhood of $z$ in $S^1$ such that, for each $z'$ in this neighborhood, $a$ and $b$ do not belong to $\sigma_{z'}$ and moreover
\begin{align*}
\sum_{\lambda\in(a,b)} \ind_{z,\lambda}(\qq)
=
\sum_{\lambda\in(a,b)} \ind_{z',\lambda}(\qq).
\end{align*}
 
Now, assume that $z\in S^1$ is not a Poincar\'e point, namely  $0$ does not belong to $\sigma_z$. Then, by the above continuity property, $0$ does not belong to $\sigma_{z'}$ for each $z'$ in some neighborhood of $z$, and moreover
\[  
\ind_z(\qq)
=
\sum_{\lambda<0} \ind_{z,\lambda}(\qq)
=
\sum_{\lambda<0} \ind_{z',\lambda}(\qq)
=
\ind_{z'}(\qq).
\]
Finally, assume that $z\in S^1$ is a Poincar\'e point, and let us fix a sufficiently small $\epsilon>0$ so that $[-\epsilon,\epsilon]\cap\sigma_z=\{0\}$. Then, by the above continuity property, $-\epsilon$ and $\epsilon$ do not belong to $\sigma_{z'}$ for each $z'$ in some neighborhood of $z$, and we have
\begin{align*}
\ind_{z'}(\qq)
& =
\sum_{\lambda<-\epsilon} \ind_{z',\lambda}(\qq)
+
\sum_{\lambda\in(-\epsilon,0)} \ind_{z',\lambda}(\qq)\\
&=
\sum_{\lambda<-\epsilon} \ind_{z,\lambda}(\qq)
+
\sum_{\lambda\in(-\epsilon,0)} \ind_{z',\lambda}(\qq)\\
&=
\ind_{z}(\qq)
+
\sum_{\lambda\in(-\epsilon,0)} \ind_{z',\lambda}(\qq).
\end{align*}
This proves that $\ind_{z}(\qq)\leq \ind_{z'}(\qq)$. Moreover
\begin{align*}
\ind_{z'}(\qq)
&=
\ind_{z}(\qq)
+
\sum_{\lambda\in(-\epsilon,0)} \ind_{z',\lambda}(\qq)\\
&\leq
\ind_{z}(\qq)
+
\sum_{\lambda\in(-\epsilon,\epsilon)} \ind_{z',\lambda}(\qq)\\
&=
\ind_{z}(\qq)
+
\sum_{\lambda\in(-\epsilon,\epsilon)} \ind_{z,\lambda}(\qq)\\
&=\ind_{z}(\qq)+\nul_{z}(\qq).
\end{align*}
The statement regarding $\coind_{z}(\qq)$ is established in the same way.
\end{proof}
\end{prop}

\begin{proof}[Proof of Proposition~\ref{p:iteration_inequalities}]
Point~(i) follows from Propositions~\ref{p:Fourier} and~\ref{p:Poincare}. As for point~(ii), let us fix $m\in\N$ and denote by $z_1,...,z_r\in S^1$ the Poincar\'e points of $\qq$, so that $z_1^m,...,z_r^m$ are the Poincar\'e points of $\qq\iter m$. By Proposition~\ref{p:ind_z,coind_z}, for each $w,z\in S^1$ we have
\begin{align}\label{e:iter_ineq_tech1}
\ind_w(\qq\iter m)+\nul_w(\qq\iter m)
\leq
\ind_z(\qq\iter m)+
\sum_{\alpha=1}^r
\nul_{z_\alpha^m}(\qq\iter m)
\end{align}
By Propositions~\ref{p:Fourier} and~\ref{p:Poincare} we have
\[ 
\sum_{\alpha=1}^r
\nul_{z_\alpha^m}(\qq\iter m)
=
\sum_{\alpha=1}^r
\nul_{z_\alpha}(\qq)
\leq 2N,
\]
and, together with~\eqref{e:iter_ineq_tech1}, we obtain
\begin{align}\label{e:iter_ineq_tech2}
\ind_w(\qq\iter m)+\nul_w(\qq\iter m)
\leq
\ind_z(\qq\iter m) +2N.
\end{align}
Now, by Propositions~\ref{p:Fourier} and~\ref{p:ind_z,coind_z}, we have
\begin{align*}
\avind(\qq)
=
\lim_{m\to\infty} \frac{\ind(\qq\iter m)}{m}
=
\lim_{m\to\infty} \frac1m \sum_{w^m=1} \ind_w(\qq)
=
\frac{1}{2\pi } \int_{0}^{2\pi} \ind_{e^{i\theta}}(\qq)\,\diff\theta.
\end{align*}
Notice that
\begin{align*}
\frac{1}{2\pi } \int_{0}^{2\pi} \ind_{e^{i\theta}}(\qq\iter m)\,\diff\theta
=
\avind(\qq\iter m)
=
m\,\avind(\qq),
\end{align*}
and moreover, since $\nul_z(\qq\iter m)=0$ for every $z\in S^1\setminus\{z_1^m,...,z_r^m\}$, we have 
\begin{align*}
\frac{1}{2\pi } \int_{0}^{2\pi} \nul_{e^{i\theta}}(\qq\iter m)\,\diff\theta
=
0.
\end{align*}
Now, by setting $z=1$ and integrating $w$ on $S^1$ in~\eqref{e:iter_ineq_tech2}, we get
\[m\,\avind(\qq)\leq\ind(\qq\iter m)+2N.\]
Then,  by setting $w=1$ and integrating $z$ on $S^1$ in~\eqref{e:iter_ineq_tech2}, we get
\[\ind(\qq\iter m)+\nul(\qq\iter m)\leq m\,\avind(\qq)+2N.\]
This proves point (ii) for the Morse index. The proof for the Morse coindex is analogous.
\end{proof}

\subsection{Local homology of iterated periodic billiard trajectories}\label{ss:loc_hom}

This section is devoted to prove that, if the Morse indices of an isolated critical point of the length functional are preserved by iteration, then the same is true for its local homology. As for the previous subsection, there is a clear parallel with the theory of closed geodesics: in fact, this result has been established for closed geodesics by Gromoll and Meyer \cite{GM_geod}, and further extended to more general Lagrangian systems by Long \cite{Lo_conley}, Lu \cite{Lu_conley} and the author \cite{Maz}. In the following, we prove the result after recalling the notion of local homology.

From now on, all the homology groups are assumed to have coefficients in $\Z_2$. For technical reasons (see the discussion after Proposition~\ref{p:FaTa}) we will work with minus the length functional, that is, with super-levels of the length functional. The \textbf{local homology} of $-\length n$ at an isolated critical point $\qq$ is the homology group $\Loc_*(\qq)$ defined by
\[ 
\Loc_*(\qq):=\Hom_*(\conf n(S)_{> c}\cup\{\qq\},\conf n(S)_{> c}),
\]
where $c=\length n(\qq)$ and $\conf n(S)_{> c}=\{ \qq'\in\conf n(S)\,|\,\length n (\qq')>c \}$. Recall that the dihedral group $\D_n$ acts by coordinates-permutations on the cyclic configuration space $\conf n(S)$ and $\length n$ is invariant under its action, see Section~\ref{s:preliminaries}. The local homology of $-\length n$ at the $\D_n$-orbit of  $\qq$ is  defined by
\[ 
\Loc_*(\D_n\cdot\{\qq\}):=\Hom_*(\conf n(S)_{> c}\cup\D_n\cdot\{\qq\},\conf n(S)_{> c}).
\]
By excision, it is straightforward to verify that the local homology   of $\D_n\cdot\{\qq\}$ is the direct sum of the local homology of each element in the $\D_n$-orbit of $\qq$, and in particular the inclusion induces a homology monomorphism
\[ \Loc_*(\qq)\hookrightarrow\Loc_*(\D_n\cdot\{\qq\}). \]
We also recall that the local homology $\Loc_k(\qq)$, and therefore $\Loc_k(\D_n\cdot\{\qq\})$ as well, is possibly nontrivial only if $\coind(\qq)\leq k\leq\coind(\qq) + \nul(\qq)$, and that $\qq$ is a local maximum of $\length n$ if and only if
\begin{align*}
\Loc_k(\qq) 
=
\left\{
  \begin{array}{lcl}
    \Z_2 & &\mbox{if } k=0, \\ 
    0 & &\mbox{if }  k\neq0. 
  \end{array}
\right.
\end{align*}
For a proof of these results as well as for more details on the local homology groups, see \cite{GM} or \cite[Chapter~I]{Ch}. The reader should keep in mind that the Morse coindex of $\length n$ is the Morse index of $-\length n$ and vice versa. 

The main result of this subsection is the following.
\begin{prop}\label{p:iter_iso}
Let $\qq\in\conf n(S)$ be an isolated critical point of $\length n$ with critical value $\length n(\qq)=c$, and assume that, for some $m\in\N$, we have
\[ \coind(\qq)=\coind(\qq\iter m),\s\s\nul(\qq)=\nul(\qq\iter m). \]
Then,   the iteration map $\itmap m$ induces the homology isomorphism
\[
\itmap m_*:\Loc_*(\qq)\toup^{\cong}\Loc_*(\qq\iter m).
\]
\end{prop}
We will prove this proposition by means of the following abstract principle of Morse theory. Here we only quote the statement for a finite dimensional setting, as needed for our purposes.

\begin{prop}[Theorem~4.1 in \cite{Maz}]\label{p:emb_iso}
Let $U\subseteq\R^k$ be an open neighborhood of the origin,  $F:U\to\R$ a smooth functional having the origin as isolated critical point, and $\VV\subset\R^k$ a vector subspace. Assume that $\nabla F(x)\in\VV$ for each $x\in U\cap\VV$, and that the Morse index and the nullity of $F$ at the origin are equal to the Morse index and the nullity of the restricted functional $F|_{U\cap\VV}$ at the origin. Then the inclusion $U\cap\VV\subset U$ induces an isomorphism between the local homology of $F|_{U\cap\VV}$ at the origin and the local homology of $F$ at the origin.
\hfill\qed
\end{prop}

\begin{proof}[Proof of Proposition~\ref{p:iter_iso}]
All we need to do is to reduce our setting in such a way that we satisfy the assumptions of Proposition~\ref{p:emb_iso}. For each $j\in\Z_n$ let us consider a chart $\phi_j:V_j\to\R^N$ for $S$, where $V_j$ is an open neighborhood of $q_j$. Up to shrink the $V_j$'s, we can assume that $V_j\cap V_{j+1}=\varnothing$. We further define smooth functions $\ell_j:V_j\times V_{j+1}\to\R$ by
\[ \ell_j(x_j,x_{j+1})=\left|\phi_j^{-1}(x_j)-\phi_{j+1}^{-1}(x_{j+1})\right|,\s\s\forall (x_j,x_{j+1})\in V_j\times V_{j+1}. \]
Then, the product $V:=V_0\times...\times V_{n-1}$ is an open set of $\conf n(S)$, and the map \[\pphi:=(\phi_0,...,\phi_{n-1}):V\to\R^{nN}\] is a chart for $\conf n(S)$ centered at $\qq$. In this local coordinates the length functional $\length n$ can be written as
\[ 
\ell\iter n(\bm x):=\length n\circ\pphi^{-1}(\bm x)=\sum_{j\in\Z_n} \ell_j(x_j,x_{j+1}),\s\s\forall\bm x=(x_0,...,x_{n-1})\in \pphi(V). 
\]
A straightforward computation shows that the gradient of $\ell\iter n$ at $\bm x$, with respect to the flat metric of $\R^{nN}$, is given by $\bm g=(g_0,...,g_{n-1})$, where
\begin{align}\label{e:gradient1}
g_j= \partial_1\ell_j(x_j,x_{j+1}) + \partial_2\ell_{j-1}(x_{j-1},x_j),\s\s\forall j\in\Z_n.
\end{align}
Now, consider the map
\[
\pphi\iter m=(\underbrace{\bigl.\pphi,...,\pphi}_{\times m}):V^{\times m}\to(\R^{nN})\iter m.
\]
This map is a chart for $\conf{nm}(S)$ centered at $\qq\iter m$, and in this local coordinates we denote the length functional $\length {nm}$ by 
$\ell\iter{nm}:=\length {nm}\circ(\pphi\iter m)^{-1}$.  Now, if we put on $(\R^{nN})\iter m$ the flat metric rescaled by the factor $m^{-1}$, the gradient of  $\ell\iter{nm}$ at $\bm y$ is given by $\tilde{\bm g}=(\tilde g_0,...,\tilde g_{nm-1})$, where
\begin{align}\label{e:gradient2}
\tilde g_j= \partial_1\ell_{j\,\textrm{mod}\,n}(y_j,y_{j+1}) + \partial_2\ell_{j-1\,\textrm{mod}\,n}(y_{j-1},y_j),\s\s\forall j\in\Z_{nm}.
\end{align}
Let us denote by $\Psi\iter m:=\pphi\iter m\circ\itmap m\circ\pphi^{-1}$ the iteration map in local coordinates, which turns out to be the restriction of the (linear) diagonal embedding of $\R^{nN}$ into $(\R^{nN})^{\times m}$. Notice that $\ell\iter{nm}\circ\Psi\iter m=m\,\ell\iter n$. From this, together with~\eqref{e:gradient1} and~\eqref{e:gradient2}, we infer
\[ (\nabla \ell\iter{nm})\circ\Psi\iter m=m\,\nabla\ell\iter n.
 \]
Therefore, the claim of the proposition follows by applying Proposition~\ref{p:emb_iso} with $U=\pphi\iter m(V^{\times m})$,  $F=-m^{-1}\ell\iter{nm}$ and the inclusion $U\cap\VV\subset U$ given by the iteration map $\Psi\iter m$.
\end{proof}

\subsection{Homological vanishing by iteration}\label{ss:hom_vanish}

In this section we show how to recover, in the ``discrete'' setting of billiards, the homological vanishing under iteration. This is a remarkable phenomenon, discovered by Bangert in \cite[Section~3]{Ban} in the study of closed geodesics and further investigated by several authors in \cite{BK, Lo_conley, Lu_conley, Maz}.

Let $\qq$ be a critical point of the length functional $\length n$ with critical value $c=\length n(\qq)$. For each $\epsilon>0$ sufficiently small, the subset $\cconf n\epsilon(S)\subset\conf n(S)$ defined in~\eqref{e:cconf} contains $\qq$, and therefore for each $m\in\N$ the subset $\cconf {nm}\epsilon(S)\subset\conf {nm}(S)$ contains $\qq\iter m$. Now notice that, by excision, the local homology of $\length n$ at $\qq$ can be equivalently defined as 
\[ \Loc_*(\qq)=\Hom_*(\cconf n\epsilon(S)_{>c}\cup\gbra\qq,\cconf n\epsilon(S)_{>c}). \]
Then, let us fix an arbitrary $b<c$ and consider the iteration map restricted as a map of pairs of the form 
\[ 
\itmap m:(\cconf n\epsilon(S)_{>c}\cup\gbra\qq,\cconf n\epsilon(S)_{>c})
\hookrightarrow
(\cconf {nm}\epsilon(S)_{>mb},\cconf {nm}\epsilon(S)_{>mc}).
\]
In homology, this map induces the homomorphism
\begin{align}
\label{e:iter_vanish}
\itmap m_*:
\Loc_*(\qq)
\to
\Hom_*(\cconf {nm}\epsilon(S)_{>mb},\cconf {nm}\epsilon(S)_{>mc}).
\end{align}

\begin{prop}[Homological vanishing by iteration]\label{p:hom_vanish}
Assume that $\qq$ is not a local maximum for $\length n$. Then, for each integer $p\geq2$, there exists $\bar m=\bar m(\qq,p)\in\N$ that is a nonnegative  power of $p$ such that the homomorphism $\itmap {\bar m}_*$  as in~\eqref{e:iter_vanish}  is the zero one.
\end{prop}

The proof is based on a homotopical result that we are going to discuss here. Consider a point $\qq=(q_0,...,q_{n-1})$ in the space $\cconf n\epsilon(S)$ and, for each $j\in\Z_n$, a sufficiently small neighborhood $U_j\subset S$ of $q_j$, so that   $U:=U_0\times...\times U_{n-1}$ is an open neighborhood  of $\qq$ in $\cconf n\epsilon(S)$. For each $c\in\R$, we denote by $U_{>c}$ the set of  points $\qq'$ in $U$ such that $\length n(\qq')>c$. Analogously, for each $m\in\N$, we denote by $U\iter m_{>c}$ the set of points $\qq''$ in $U\iter m$ such that $\length {nm}(\qq'')>c$. Notice that the iteration map $\itmap m$ send $U_{>c}$ into $U\iter m_{>mc}$. Then, for each $j\in\N$, we denote by $\Delta^j$ the standard $j$-simplex in $\R^{j}$.

\begin{lem}[Homotopical vanishing by iteration]\label{l:hom_vanish}
Let $b<c$ such that $U_{>c}\neq\varnothing$, and consider the singular simplex $\sigma:(\Delta^j,\partial\Delta^j)\to(U_{>b},U_{>c})$, i.e.\ $[\sigma]\in\pi_j(U_{>b},U_{>c})$. Then, there exists $\bar m=\bar m(\sigma)\in\N$ and, for each integer $m\geq\bar m$, a homotopy 
\begin{align}\label{e:bangert_form}
\sigma\biter m: [0,1]\times(\Delta^j,\partial\Delta^j)\to(U\iter m_{>mb},U\iter m_{>mc})
\end{align}
such that
\begin{itemize}
\item[(i)] $\sigma\biter m(0,\cdot)=\sigma\iter m:=\itmap m\circ\sigma$,
\item[(ii)] $\sigma\biter m(t,x)=\sigma\biter m(0,x)$ for each $x\in\partial\Delta^j$,
\item[(iii)] $\sigma\biter m(1,\Delta^j)\subset U\iter m_{>mc}$.
\end{itemize}
In particular $[\sigma\iter m]=0$ in $\pi_j(U\iter m_{>mb},U\iter m_{>mc})$.
\begin{proof}
Let us begin by explaining the basic construction that will be employed in the proof. Consider a map 
$\gamma:[x_0,x_1]\to U$, where $[x_0,x_1]\subset\R$.  For each $m\in\N$ we define a map 
\[\gamma\bang m:[x_0,x_1]\to U\iter m\] 
in the following way: for each $k\in\{0,...,m-1\}$ and $y\in[x_0,x_1]$, denoting $\qq_0=\gamma(x_0)$, $\qq_1=\gamma(x_1)$ and $x=x_0 + (x_1-x_0)\sfrac{k}{m}+\sfrac{y-x_0}{m}$, we set 
\[
\gamma\bang m(x)
:=
(
\underbrace{\bigl.\qq_1,...,\qq_1}_{\times k},
\gamma(y),\underbrace{\qq_0,...,\qq_0\bigr.}_{\times m-k-1}
).
\]
The length of $\gamma\bang m(x)$ is bounded from below as
\begin{equation}\label{e:estimate_bangert}
\begin{split}
\length{nm}(\bm\gamma\bang m(x))
&
\geq
(k-1) \length n(\qq_1)
+
(m-k-2) \length n(\qq_0)\\
&
\geq
(m-3) \min\!\gbra{ \length{n}(\bm\gamma(x_0)),\length{n}(\bm\gamma(x_1)) }
.
\end{split}
\end{equation}

Now, consider the singular simplex $\sigma$ of the statement. We want to decompose its domain $\Delta^j$ as a ``continuous'' family of paths, and then apply the above construction to each path separately. Let $\Line\subseteq\R^j$ be the axis passing through the origin and the barycenter of   $\Delta^j\subset\R^j$. For each $s\in[0,1]$ we denote by $s\Delta^j$ the rescaled $j$-simplex given by $\{s z\,|\, z\in\Delta^j\}$.   For each $z\in s\Delta^j$, we denote by $[x_0(s,z),x_1(s,z)]\subset s\Delta^j$ the maximum segment that contains $z$ and is parallel to $\Line$. This notation is summarized, for $j=2$, in Figure~\ref{f:simplex}.

\frag[s]{a}{$\underbrace{\hspace{72pt} }_{\mbox{$s$}}$}%
\frag[s]{L}{$\Line$}%
\frag[ss]{x0}{$x_0(s,z)$}%
\frag[ss]{x1}{$x_1(s,z)$}%
\frag[ss]{z}{$z$}%
\begin{figure}
\begin{center}
\includegraphics{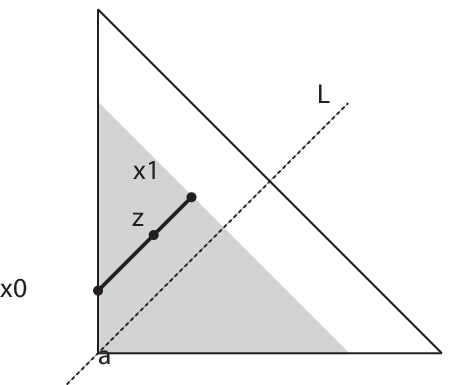}
\end{center}
\captionstyle{myCenter}
\caption{}
\label{f:simplex}
\end{figure}

We define the homotopy $\sigma\biter m: [0,1]\times\Delta^j\to U$ by
\begin{align*}
\sigma\biter m(s,z)
=
\left\{
  \begin{array}{lll}
    \Bigl.(\sigma|_{[x_0(s,z),x_1(s,z)]})\bang m(z) &  & \mbox{if }z\in s\Delta^j, \\ 
    \Bigl.\sigma\iter m(z) &  & \mbox{if }z\not\in s\Delta^j. 
  \end{array}
\right.
\end{align*}
This homotopy clearly satisfies properties (i-ii) in the statement. Then, for each $s\in[0,1]$ and $z\in s\Delta^j$, by the estimate in~\eqref{e:estimate_bangert} we have
\begin{align*}
\length{nm}(\sigma\biter m(s,z))
\geq\, &
(m-3) \min\!\gbra{ \length{n}(\sigma(x_0(s,z))),\length{n}(\sigma(x_1(s,z))) },
\end{align*}
while for each $z\in \Delta^j\setminus s\Delta^j$ we have
\begin{align*}
\length{nm}(\sigma\biter m(s,z))
=
\length{nm}(\sigma\iter m(z))
=
m\,\length n(\sigma(z)).
\end{align*}
Let us choose $\delta>0$ sufficiently small so that $\sigma(\Delta^j)\subset U_{>b+\delta}$ and $\sigma(\partial\Delta^j)\subset U_{>c+\delta}$. For each $(s,z)\in[0,1]\times\Delta^j$ we obtain
\begin{align*}
\length{nm}(\sigma\biter m(s,z))
\geq\, &
(m-3) (b+\delta),\\
\length{nm}(\sigma\biter m(1,z))
\geq\, &
(m-3) (c+\delta).
\end{align*}
This implies that, for $m\in\N$ sufficiently big, the homotopy $\sigma\biter m$ has the form~\eqref{e:bangert_form} and satisfies (iii).
\end{proof}
\end{lem}

\begin{proof}[Proof of Proposition~\ref{p:hom_vanish}]
The proof is based on the homotopical vanishing principle of Lemma~\ref{l:hom_vanish}, together with the homotopical invariance property of singular homology as stated in~\cite[Lemma~1]{BK}. 

Let us fix the integer $p\geq2$ of the statement and set \[\K_{p}=\{p^{n}\,|\,n\in\N\cup\{0\}\}.\] Since the local homology group $\Loc_{*}(\qq)$ is finitely generated (see~\cite[Lemma~2]{GM}), all we need to prove is that, for each homology class $[\mu]\in\Loc_{*}(\qq)$ and for some $\bar m=\bar m([\mu])\in\K_{p}$, we have $\itmap {\bar m}_{*}[\mu]=0$ in $\Hom_*(\cconf {n \bar m}\epsilon(S)_{>\bar{m}b},\cconf {n \bar m}\epsilon(S)_{>\bar{m}c})$.

Let $U$ be the open neighborhood of $\qq$ introduced in the paragraph preceding  Lemma~\ref{l:hom_vanish}. By excision, we can assume that the relative cycle $\mu$ representing $[\mu]$ has support contained in $U$. We  denote by $\Sigma(\mu)$ the collection of singular simplices in $\mu$ and all their lower dimensional faces. For each $\sigma:\Delta^j\to U$ contained in $\Sigma(\mu)$ we will define $\bar m=\bar m(\sigma)\in\K_p$ and a homotopy 
\[\sigma\siter{\bar m}:[0,1]\times \Delta^j \to U\iter{\bar m}_{>\bar mb}\]
such that
\begin{itemize}
\item[(i)] $\sigma\siter{\bar m}(0,\cdot)=\sigma\iter {\bar m}$,
\item[(ii)] $\sigma\siter{\bar m}(1,\Delta^j)\subset  U\iter{\bar m}_{>\bar mc}$,
\item[(iii)] if $\sigma(\Delta^j)\subset U_{>c}$, then  $\sigma\siter{\bar m}(s,\cdot)=\sigma\iter{\bar m}$ for each $s\in[0,1]$,
\item[(iv)] $(\sigma\circ F_i)\siter{\bar m}=\sigma\siter{\bar m}(\cdot,F_i(\cdot))$ for each $i\in\gbra{0,...,j}$, where $F_i:\Delta^{j-1}\to\Delta^j$ is the standard affine map onto the $i^{\mathrm{th}}$ face of $\Delta^j$.
\end{itemize}
For each $m\in\K_p$ greater than $\bar{m}$, we define a homotopy $\sigma\siter{m}:[0,1]\times \Delta^j \to U\iter{m}_{>mb}$ by \[\sigma\siter{m}:=\itmap {m/\bar m}\circ\sigma\siter{\bar m}.\] This homotopy satisfies the analogous properties (i-iv) in period $m$. Notice that property (iv) implicitly requires that $\bar m(\sigma\circ F_i)\leq\bar m(\sigma)$ for each $i\in\gbra{0,...,j}$. 

Now, take a sufficiently big $m\in\K_p$ so that $\sigma\siter m$ is defined for all $\sigma\in\Sigma(\mu)$. Then, by means of the above homotopies, the relative cycle $\mu\iter m=\itmap m\circ \mu$ is homologous to a relative cycle whose support is contained in $U\iter m_{>mc}\subset\cconf m\epsilon(S)_{>mc}$. In fact, this latter relative cycle is obtain from $\mu\iter m$ by homotoping each singular simplex $\sigma\iter m\in\Sigma(\mu\iter m)$ to $\sigma\siter m(1,\cdot)$ via the homotopy $\sigma\siter m$. This implies that $\itmap m_*[\mu]=[\mu\iter m]=0$ in $\Hom_*(\cconf {n  m}\epsilon(S)_{>mb},\cconf {n  m}\epsilon(S)_{>mc})$. In order to finish the proof, we only need to build the above homotopies satisfying (i-iv). The idea is to apply Lemma~\ref{l:hom_vanish} subsequently to all the faces of the singular simplices in $\Sigma(\mu)$, starting from the 0-dimensional faces and then going up to the higher dimensional ones. 

Let us  proceed by induction, starting by assuming that $\mu$ is a $0$-relative cycle. Then, $\Sigma(\mu)$ is simply a finite set of points in $U_{>b}$. Let $\ww$ be one of these points. If $\ww\in U_{>c}$ we are already done: we simply  set $\bar m=\bar m(\ww)=1$ and $\ww\siter{\bar m}(s)=\ww$ for each $s\in[0,1]$. In the other case, $\ww\in U_{>b}\setminus U_{>c}$, we take an arbitrary continuous path $\gamma:[0,1]\to U$ such that $\gamma(0)=\ww$ and $\gamma(1)\in U_{>c}$. Notice that such a path exists, since the critical point $\qq$ of the statement is not a local minimum and therefore, since $\length n(\qq)=c$, the subset $U_{>c}$ is not empty. Then, consider $\bar m=\bar m(\gamma)\in\N$ and the associated homotopy 
\[\gamma\biter{\bar m}:[0,1]\times[0,1]\to U\iter{\bar m}\] 
given by Lemma~\ref{l:hom_vanish}, such that $\gamma\biter{\bar m}(s,0)=\ww\iter{\bar m}$, $\gamma\biter{\bar m}(s,1)=\gamma(1)\iter{\bar m}$ and $\gamma\biter{\bar m}(1,s)\in U\iter{\bar m}_{>\bar m b}$ for each $s\in[0,1]$. Without loss of generality we can assume that $\bar m\in\K_p$, and we set 
\[\ww\siter {\bar m}=\gamma\biter{\bar m}(1,\cdot).\]

When $\mu$ is a $j$-relative cycle, with $j\geq 1$, we can apply the inductive hypothesis: for every  nonnegative integer $j'\leq j-1$ and for each $j'$-singular simplex $\sigma\in\Sigma(\mu)$ we obtain  $\bar m=\bar m(\sigma)\in\K_p$ and  a homotopy $\sigma\siter {\bar m}$ satisfying the above properties (i-iv). Now, consider a $j$-singular simplex $\sigma\in\Sigma(\mu)$. If $\sigma(\Delta^j)\subset U_{>c}$ we simply  set $\bar m=\bar m(\sigma):=1$ and $\sigma\siter{\bar m}(s,\cdot):=\sigma$ for each $s\in[0,1]$. Otherwise, if $\sigma(\Delta^j)\not\subset  U_{>c}$,  we denote by $\bar m'=\bar m'(\sigma)$ the maximum of the $\bar m(\nu)$'s for all the proper faces $\nu$ of $\sigma$. For each   $m\in\K_p$ greater than or equal to $\bar m'$,  every such $\nu$  has an associated  homotopy   
$\nu\siter m$ satisfying the above condition (i-iv).
 For technical reasons, let us assume that $\nu\siter {m}(s,\cdot)=\nu\siter m(\sfrac12,\cdot)$ for each  $s\in[\sfrac12,1]$. Patching together the homotopies of the proper faces of $\sigma$, we obtain
\[  \sigma\siter m: ([0,\sfrac12]\times \partial\Delta^j) \cup (\gbra0\times \Delta^j) \to U\iter m_{>mb},\]
such that $\sigma\siter m(0,\cdot)=\sigma\iter m$ and $\sigma\siter m(\cdot,F_i(\cdot))=(\sigma\circ F_i)\siter m$ for each $i=0,...,j$. By retracting $[0,\sfrac12]\times \Delta^j$ onto $([0,\sfrac12]\times \partial\Delta^j) \cup (\gbra0\times \Delta^j)$ we can extend  the homotopy $\sigma\siter m$, obtaining
\begin{align}\label{e:P_homot_1}  
\sigma\siter m: [0,\sfrac12]\times \Delta^j \to U\iter m_{>mb}.
\end{align}
Notice that $\sigma\siter {\bar m'}(\sfrac12,\cdot)$ is a singular simplex of the form
\[
\sigma\siter {\bar m'}(\sfrac12,\cdot): (\Delta^j,\partial\Delta^j) \to (U\iter{\bar m'}_{>\bar m'b},U\iter{\bar m'}_{>\bar m'c} ).
\]
Let us briefly denote this singular simplex by $\tilde\sigma$, and consider $\bar m''=\bar m(\tilde\sigma)$ and the homotopy $\tilde\sigma\biter{\bar m''}$ given by Lemma~\ref{l:hom_vanish}, so that 
\begin{align*}
& \tilde\sigma\biter{\bar m''}(0,\cdot)=\tilde\sigma\iter {\bar m''}=\sigma\siter {\bar m''\bar m'}(\sfrac12,\cdot),\\
& \tilde\sigma\biter{\bar m''}(1,\Delta^j)\subset U\iter{\bar m'\bar m''}_{>\bar m'\bar m''c}.
\end{align*} 
Then, we set $\bar m=\bar m(\sigma):=\bar m'\bar m''$ and we extend the homotopy in~\eqref{e:P_homot_1} to $[0,1]\times\Delta^j$ by 
\[
\sigma\siter{\bar m}(s,\cdot):= \tilde\sigma\biter {\bar m''}(2s-1,\cdot),\s\s\forall s\in[\sfrac12,1]. \qedhere
\]
\end{proof}

\section{Proof of the main result}\label{s:proof}

In this section we carry out the proof of Theorem~\ref{t:main}, stated in the introduction. Let us adopt the notation of Section~\ref{s:preliminaries}, so that our billiard table is the strictly convex compact subset $U_S$ enclosed by a smooth hypersurface $S\subset\R^{N+1}$, with $N\geq2$. The proof of Theorem~\ref{t:main} will be by contradiction: let us fix, once for all, a prime number $p\in\N$ and let us assume that the following two conditions hold:
\begin{itemize}
\item[\textbf{(F1)}] 
for each $n\in\N$, there are only finitely many periodic billiard trajectories bouncing $p^n$ times;
\item[\textbf{(F2)}] there are only finitely many periodic billiard trajectories $\gamma_1,...,\gamma_r$ satisfying properties (i-ii) in the statement.
\end{itemize}
Let us define $\K_p:=\{p^{n}\,|\,n\in\N\cup\{0\}\}$. For each $\alpha\in\{1,...,r\}$, let us denote by  $\tilde\qq_\alpha=(\tilde q_{\alpha,1},...,\tilde q_{\alpha,n_\alpha})\in\conf {n_\alpha}(S)$  the sequence of bounce points of $\gamma_\alpha$. Let  $n\in\K_p$  be the maximum of the $n_\alpha$'s, and let us set 
\[\qq_\alpha:=\tilde\qq_\alpha\iter{n/n_\alpha}\in\conf n(S),\s\s\forall\alpha\in\{1,...,r\}.\]

\begin{rem}\label{r:only_q_alpha}
For each critical point $\qq$ of the length functional $\length n$, the function $m\mapsto\coind(\qq\iter m)$, where $m\in\K_p$, is monotone increasing (although not strictly). In fact, the differential of the iteration map $\itmap m$ at $\qq$ maps (injectively) any positive eigenspace of $\hess\length n(\qq)$ to a positive eigenspace of $\hess\length{nm}(\qq\iter m)$. This implies that, by Assumption~\textbf{(F2)}, for all $m\in\K_p$ the set of critical points of $\length {nm}$ with Morse coindex less than or equal to $N$ is given by the $\D_{nm}$-orbits of those $\qq_\alpha\iter m$'s such that $\coind(\qq_\alpha\iter m)\leq N$.
\hfill\qed
\end{rem}

As a first step in our proof, let us establish the following claim. We refer the reader to Section~\ref{ss:loc_hom} for the definition and properties of  local homology groups. From now on, all the homology groups are assumed to have coefficients in $\Z_2$.

\begin{claim}\label{claim:1}
There exists $\qq\in\{\qq_1,...,\qq_r\}$ with $\avcoind(\qq)=0$ and an infinite subset $\K_p'\subset\K_p$ such that, for each $m\in\K_p'$, the local homology group $\Loc_{N-1}(\qq\iter m)$ is nontrivial.
\begin{proof}
For each $m\in\K_p$, let us fix a sufficiently small $\epsilon=\epsilon(m)>0$  and consider the space $\cconf {nm}\epsilon(S)$ introduced in Section~\ref{s:preliminaries}. By Proposition~\ref{p:FaTa}(iii),  this space contains all the critical points   of $\length {nm}$. Proposition~\ref{p:FaTa}(ii) implies that the inclusion $\cconf {nm}\epsilon(S)\hookrightarrow\conf {nm}(S)$ induces an isomorphism in homology, and therefore by~\eqref{e:Poincare_polynomial}  we infer that 
\begin{align}\label{e:homology_epsilon}
\Hom_{N-1}(\cconf {nm}\epsilon(S))\neq0.
\end{align} 
By Proposition~\ref{p:FaTa}(i,iv), the functional $-\length{nm}:\cconf {nm}\epsilon(S)\to\R$ satisfies the ``general boundary conditions'' for Morse theory (see \cite[Section~6.1]{Ch}), and we have the Morse inequality
\begin{align}\label{e:Morse_inequality}
\dim \Hom_k(\cconf {nm}\epsilon(S))
\leq
\sum_{\qq'} 
\dim \Loc_k(\qq'),
\end{align}
where the sum in the right hand side runs over all the critical points $\qq'$ of $\length{nm}$. Notice that only those $\qq'$ such that 
\begin{align}
\label{e:claim1_qq'}
\coind(\qq')\leq k\leq\coind(\qq')+\nul(\qq')
\end{align} 
may possibly give a nonzero contribution. Now, choosing $k=N-1$, by  Remark~\ref{r:only_q_alpha} the elements in the $\D_{nm}$-orbit of the $\qq_\alpha\iter m$'s are the only critical points $\qq'$ of $\length {nm}$ that may satisfy~\eqref{e:claim1_qq'}. Hence, by~\eqref{e:homology_epsilon} and by the Morse inequality~\eqref{e:Morse_inequality} we infer
\begin{align}\label{e:claim1_q_alpha}
0\neq 
\sum_{\alpha=1}^r 
\dim \Loc_{N-1}(\D_{nm}\cdot\{\qq_\alpha\iter m\}).
\end{align}
Now, if all the $\qq_\alpha$'s had nonzero mean Morse coindex $\avcoind(\qq_\alpha)$, by the iteration inequality in Proposition~\ref{p:iteration_inequalities}(ii) we would have $\coind(\qq_\alpha\iter m)>N-1$, providing $m\in\K_p$ is big enough. However, this would imply that   $\Loc_{N-1}(\D_{nm}\cdot\{\qq_\alpha\iter m\})=0$  for each $\alpha\in\{1,...,r\}$, contradicting~\eqref{e:claim1_q_alpha}. Hence, some of the $\qq_\alpha$'s, say $\qq_1,...,\qq_s$ where $s\leq r$, satisfy $\avcoind(\qq_\alpha)=0$. Up to choosing $m$ big enough the inequality in~\eqref{e:claim1_q_alpha} reduces to
\begin{align*}
0\neq 
\sum_{\alpha=1}^s 
\dim \Loc_{N-1}(\D_{nm}\cdot\{\qq_\alpha\iter m\}).
\end{align*}
This  implies that at least one $\qq$ among $\qq_1,...,\qq_s$ satisfies $\Loc_{N-1}(\D_{nm}\cdot\{\qq\iter m\})\neq0$, and thus $\Loc_{N-1}(\qq\iter m)\neq0$ as well, for infinitely many $m\in\K_p$.
\end{proof}
\end{claim}

Now, by Proposition~\ref{p:iteration_inequalities}, the indices $\coind(\qq\iter m)$ and $\nul(\qq\iter m)$ are uniformly bounded for all $m\in\K_p'$, and hence we can choose an infinite subset $\K_p''\subset\K_p'$ such that  $\coind(\qq\iter m)$ and $\nul(\qq\iter m)$ are constant in $m\in\K_p''$. Without loss of generality, let us assume that $1$ belongs to $\K_p''$ (equivalently, set $m:=\min\K_p''$ and rename $\qq$ to be $\qq\iter m$ and $\K_p''$ to be $m^{-1}\K_p''$).

We set $c:=\length n(\qq)$ and we fix an arbitrary real number $b<c$ such that none of the $\qq_\alpha$'s has critical value in the open interval $(b,c)$, i.e. $\length n(\qq_\alpha)\not\in(b,c)$ for all $\alpha\in\gbra{1,...,r}$. Now, for each $m\in\K_p''$, let us choose $\epsilon=\epsilon(m)>0$ small enough so that $\cconf{n}\epsilon(S)$ and  $\cconf{nm}\epsilon(S)$ satisfy the assertions (i-iv) of Proposition~\ref{p:FaTa}. Then, let us  consider the iteration map restricted as a map of pairs of the form 
\begin{align}\label{e:iter_restr}
\itmap m:(\cconf n\epsilon(S)_{>c}\cup\gbra\qq,\cconf n\epsilon(S)_{>c})
\hookrightarrow
(\cconf {nm}\epsilon(S)_{>mb},\cconf {nm}\epsilon(S)_{>mc}).
\end{align}

\begin{claim}\label{claim:2}
For each $m\in\K_p''$, the iteration map in~\eqref{e:iter_restr} is injective in degree $(N-1)$ homology, i.e.   
\[ 
\itmap m_*:\Loc_{N-1}(\qq)\hookrightarrow\Hom_{N-1}(\cconf {nm}\epsilon(S)_{>mb},\cconf {nm}\epsilon(S)_{>mc}),\s\s\forall m\in\K_p'' 
\]
\begin{proof}
Let us consider an arbitrary $m\in\K_p''$. By Assumption~\textbf{(F1)}, we can choose $b'\in(mb,mc)$ sufficiently close to $mc$ such that the only critical value of $\length{nm}$ in the interval $(b',mc]$ is $mc$. By Morse theory (see e.g.~\cite[Theorem~4.2]{Ch}), the inclusion
\begin{align*}
(\cconf {nm}\epsilon(S)_{>c}\cup\{\qq\iter m\},\cconf {nm}\epsilon(S)_{>c})
\hookrightarrow
(\cconf {nm}\epsilon(S)_{>b'},\cconf {nm}\epsilon(S)_{>mc})
\end{align*}
is injective in homology, namely it induces the monomorphism
\begin{align}\label{e:incl_iso_loc}
\Loc_*(\qq\iter m)
\hookrightarrow
\Hom_*(\cconf {nm}\epsilon(S)_{>b'},\cconf {nm}\epsilon(S)_{>mc}).
\end{align}
Now, let us examine the long exact sequence of the triple \[(\cconf {nm}\epsilon(S)_{>mb},\cconf {nm}\epsilon(S)_{>b'},\cconf {nm}\epsilon(S)_{>mc}),\] which we can write as the following exact triangle:
\begin{align*}
\xymatrix{
\Hom_*(\cconf {nm}\epsilon(S)_{>b'},\cconf {nm}\epsilon(S)_{>mc})
\ar[ddrr]
&&
\\\\
\Hom_*(\cconf {nm}\epsilon(S)_{>mb},\cconf {nm}\epsilon(S)_{>b'})
\ar[uu]^{\partial_*}
&&
\Hom_*(\cconf {nm}\epsilon(S)_{>mb},\cconf {nm}\epsilon(S)_{>mc})
\ar[ll]
}
\end{align*}
Here, $\partial_*$ is the boundary homomorphism which  lower the grade $*$ by one, and the other homomorphisms are simply induced by inclusions. Now, by our choice of $b$ and $b'$, none of the $\qq_\alpha\iter m$'s has critical value inside the interval $(mb,b']$ and, by assumption~\textbf{(F2)}, all the critical values of $\length {nm}$ inside $(mb,b']$ correspond to critical points that either are local maxima or have Morse coindex greater than $N$. This, together with the Morse inequalities, implies that the group $\Hom_j(\cconf {nm}\epsilon(S)_{>mb},\cconf {nm}\epsilon(S)_{>b'})$ is trivial in degree $j=N-1$ and $j=N$. Hence, the above exact triangle implies that the diagonal homomorphism is an isomorphism in degree $N-1$, i.e.\ the inclusion induces the homology isomorphism
\[
\Hom_{N-1}(\cconf {nm}\epsilon(S)_{>b'},\cconf {nm}\epsilon(S)_{>mc})
\toup^{\cong}
\Hom_{N-1}(\cconf {nm}\epsilon(S)_{>mb},\cconf {nm}\epsilon(S)_{>mc}).
\]
This isomorphism and the monomorphism in~\eqref{e:incl_iso_loc} fit in the following commutative diagram, where all the homomorphisms are induced by  inclusions:
\begin{align*}
\xymatrix{
&
\Hom_{N-1}(\cconf {nm}\epsilon(S)_{>b'},\cconf {nm}\epsilon(S)_{>mc})
\ar[dd]^{\cong}
\\
\Loc_{N-1}(\qq\iter m)
\ar@{^{(}->}[ur]
\ar[dr]^{\iota}
&
\\
&
\Hom_{N-1}(\cconf {nm}\epsilon(S)_{>mb},\cconf {nm}\epsilon(S)_{>mc})
}
\end{align*}
This forces the inclusion-induced homomorphism $\iota$ to be injective, i.e.
\[
\iota:\Loc_{N-1}(\qq\iter m)
\hookrightarrow
\Hom_{N-1}(\cconf {nm}\epsilon(S)_{>b'},\cconf {nm}\epsilon(S)_{>mc}).
\]
Now, by Proposition~\ref{p:iter_iso},  the iteration map induces the homology isomorphism
\[ 
\itmap {m}_*: 
\Loc_*(\qq)
\toup^{\cong}
\Loc_*(\qq\iter {m}).
\]
By composing the monomorphism $\iota$ with the isomorphism  $\itmap m_*$ we obtain the claim.
\end{proof}
\end{claim}

For each $m\in\K_p''$, the critical point $\qq\iter m$ of $\length{nm}$ is not a local maximum. In fact, $\Loc_{N-1}(\qq\iter m)$ is nontrivial whereas the local homology groups of a local maximum are nontrivial only in degree zero (we recall that $N\geq2$). Therefore, by the homological vanishing principle in Proposition~\ref{p:hom_vanish}, the homomorphism 
\[\itmap m_*:\Loc_*(\qq)\to\Hom_*(\cconf {nm}\epsilon(S)_{>mb},\cconf {nm}\epsilon(S)_{>mc})\]
is zero provided $m\in\K_p''$ is sufficiently big. This contradicts Claim~\ref{claim:2}, and concludes the proof of Theorem~\ref{t:main}.

\bibliography{billiards}
\bibliographystyle{amsalpha}

\vspace{0.5cm}

\end{document}